\documentclass[11pt,twoside]{article}
\usepackage{amsmath,amsfonts,amscd,amsthm}

\newtheorem{thm}{Theorem}[subsection]
\newtheorem{thm0}{Theorem}
\newtheorem{prop}[thm]{Proposition}
\newtheorem{lem}[thm]{Lemma}
\newtheorem{cor}[thm]{Corollary}

\theoremstyle{definition}
\newtheorem{defin}[thm]{Definition}

\theoremstyle{remark}
\newtheorem{rem}[thm]{Remark}
\newtheorem{ex}[thm]{Example}

\numberwithin{equation}{section}
\newcommand{\x}{\times}
\newcommand{\ox}{\otimes}
\newcommand{\dt}{.}
\def\lacute{\mathopen{<}}\def\racute{\mathopen{>}}
\newcommand{\scal}[2]{{\lacute#1,#2\racute}}
\newcommand{\ensemble}[2]{\{\,#1\mid#2\,\}}

\newcommand{\C}{{\mathbb C}}
\newcommand{\R}{{\mathbb R}}
\newcommand{\CB}{\mathcal B}
\newcommand{\CaD}{\mathcal D}   
\newcommand{\CF}{\mathcal F}
\newcommand{\CI}{\mathcal I}
\newcommand{\CJ}{\mathcal J}
\newcommand{\CK}{\mathcal K}
\newcommand{\CM}{\mathcal M}
\newcommand{\CN}{\mathcal N}
\newcommand{\CO}{\mathcal O}
\newcommand{\CT}{\mathcal T}
\newcommand{\kN}{\mathfrak N}
\newcommand{\kO}{\mathfrak O}
\newcommand{\kc}{\mathfrak c}
\newcommand{\kg}{\mathfrak g}
\newcommand{\kh}{\mathfrak h}
\newcommand{\kl}{\mathfrak l}
\newcommand{\km}{\mathfrak m}
\newcommand{\kp}{\mathfrak p}
\newcommand{\kq}{\mathfrak q}
\newcommand{\ks}{\mathfrak s}

\newcommand{\kgp}{\kg_\gp}
\newcommand{\kqp}{\kq_\gp}
\newcommand{\khp}{\kh_\gp}
\newcommand{\khpo}{\khp^\bot}
\newcommand{\khpp}{(\khpo)'}
\newcommand{\sld}{\ks\kl_2}
\newcommand{\sln}{\ks\kl_n}
\newcommand{\gln}{\kg\kl_n}
\newcommand{\Gln}{Gl_n}
\newcommand{\Sln}{Sl_n}

\newcommand{\vespa}{\vspace{1em}}
\renewcommand{\setminus}{-}     
\newcommand{\dsur}[1]{\frac \partial{\partial#1} }               
\newcommand{\OO}{{{\mathcal O_\Omega}}}
\newcommand{\Dgg}{\CaD_\kg^G}
\newcommand{\CHM}{{Ch(\CM)}}

\newcommand{\ga}{\alpha}
\newcommand{\gb}{\beta}
\newcommand{\gd}{\delta}           \newcommand{\gD}{\Delta}
\newcommand{\gz}{\zeta}
\newcommand{\gh}{\eta}
\newcommand{\gth}{\theta}          
\newcommand{\gl}{\lambda}          \newcommand{\gL}{\Lambda}
\newcommand{\gx}{\xi}              
\newcommand{\gp}{\pi}              
\newcommand{\gs}{\sigma}           \newcommand{\gS}{\Sigma}
\newcommand{\gt}{\tau}
\newcommand{\gf}{\varphi}          \newcommand{\gF}{\Phi}
\newcommand{\gy}{\psi}             \newcommand{\gY}{\Psi}
           \newcommand{\gO}{\Omega}

\newcommand{\exacte}[3]{0\xrightarrow{\ \ }{#1}\xrightarrow{\ \ }{#2} \xrightarrow{\ \ }{#3} \xrightarrow{\ \ }0}

\newcommand{\Dg}{\CaD_\kg}
\newcommand{\Dgo}{\CaD_{\kg\x \gO}}
\newcommand{\OV}{\CO_V}
\newcommand{\OU}{\CO_U}
\newcommand{\DV}{\CaD_V}
\newcommand{\DU}{\CaD_U}
\newcommand{\DO}{\CaD_\gO}
\newcommand{\MF}{\CM_F}
\newcommand{\MFG}{\CM_{F,\kg}}
\newcommand{\MG}{\CM_\kg}
\newcommand{\MFP}{\CM_{F\!,\,\kp}}

\newcommand{\MP}{\CM_\kp}
\newcommand{\MFl}{\CM_\gl^{\mathcal F}}
\newcommand{\NFG}{\CN_{F,\kg}}

\begin{document}

\title{Kirillov's conjecture and $\CaD$-modules}
\author{Esther Galina
\ and Yves Laurent}

\maketitle

\thispagestyle{empty}
\
\vskip 2cm

\section*{Introduction}

Let $G=\Gln(\R)$ or $G=\Gln(\C)$ and let $P$ be the subgroup of matrices whose last row is $(0,0,
\dots,0,1)$. Kirillov \cite{KIRILLOV} made the following conjecture:

{\bf Conjecture}
{\sl If $\pi$ is an irreducible unitary representation of $G$ on a Hilbert space $H$ then
$\pi|_P$ is irreducible.}

The proof of this conjecture has a long story, we refer to the introduction of Baruch \cite{BARUCH} for details about it.
A first proof for the complex case was done by Sahi \cite{SAHI}. The complete proof, that includes the real and complex case,
 was given by Baruch \cite{BARUCH}. He uses an argument of Kirillov to show that the conjecture is an easy corollary of the following
 theorem:
\begin{thm0}\label{thm:invariant}
Let $T$ be a $P$-invariant distribution on $G$ which is an eigendistribution with respect to the
 center of the universal enveloping algebra associated with $G$. Then there exists a locally integrable
 function $f$ on $G$ which is $G$-invariant and real analytic on the regular set $G'$ such that $T=f$. In
particular $T$ is $G$-invariant.
\end{thm0}

Barush's proof of theorem \ref{thm:invariant} uses standard methods to reduce the problem to
 nilpotent points and then needs a rather long and detailed study of the nilpotent $P$-orbits of the
adjoint representation of $P$ on the Lie algebra $\kg$ of $G$.

If we replace "$P$-invariant" by "$G$-invariant" in theorem  \ref{thm:invariant}, we get a well known
result of Harish-Chandra that we proved in \cite{GL} by means of $\CaD$-modules. We defined a class of $\CaD$-modules
that we called "tame": a $\CaD$-module is tame if it satisfies a condition on the roots of a family
of polynomials, the $b$-functions (see \S \ref{sec:tame}).  The main property of these $\CaD$-modules is that their solutions
are always locally integrable. Then we proved that in the Harish-Chandra case, the distribution $T$ is solution of a $\CaD$-module,
 i.e. a system of partial differential equations, which  is tame.

In this paper, we want to prove theorem \ref{thm:invariant} by the same method. In fact our proof will
 be simple as we will not have to calculate the roots of the $b$-functions as in \cite{GL} but
use only geometric considerations on the characteristic variety of the $\CaD$-module.
We don't need neither a concrete characterization of nilpotent $P$-orbits in $\kg$, we only use the stratification
of $\kg$ in $G$-orbits and the parametrization by the dimension of $P$-orbits in a single $G$-orbit.

Our theorem is purely complex, its is a result for $\CaD$-modules on $\Gln(\C)$. So it gives results
for distributions on any real form of $\Gln(\C)$. In the real form is $\Gln(\R)$ or $\Gln(\C)$ it gives
theorem \ref{thm:invariant}. For other real forms it gives a result on distributions which are not characterized
by the action of a group $P$ and does not seem to have an easy interpretation.

From the theorem with $G=\Gln(\C)$ we deduce easily the same theorem for $G=\Sln(\C)$ and $P$ a maximal parabolic subgroup. This
 gives the analog of theorem \ref{thm:invariant} for $\Sln(\C)$ and $\Sln(\R)$.

In section 1, we recall the definition of tame $\CaD$-modules and we define precisely the modules $\MFP$ that we want to consider. Then in section 1.3.
we state our main results. In section 2, we study the very simple but illuminating case of $\sld$.

In section 3, we prove general theorems on $\CaD$-modules defined on semi-simple Lie groups which will be used later to reduce
the dimension of the Lie algebra. Then we give the proof of the main results in section 4.

\section{Notations and definitions.}

\subsection{Tame $\CaD$-modules.}\label{sec:tame}

Let $\gO$ be a complex analytic manifold. We denote by $\OO$ the sheaf of holomorphic functions on $\gO$
and by $\DO$ the sheaf of differential operators on $\gO$ with coefficients in $\OO$. If $(x_1,
\dots,x_n)$ are local coordinates for $\gO$, we denote by $D_{x_i}$ the derivation $\dsur{x_i}$. We refer
to \cite{BJORK} for the theory of $\DO$-modules.

In this paper, we will consider coherent cyclic $\CaD$-modules that is $\CaD$-modules $\CM=\DO/\CI$ quotient of $\DO$ by  a locally finite ideal $\CI$o f $\DO$. Then the characteristic variety of $\CM$ is the subvariety of $T^*\gO$
defined by the principal symbols of the operators in $\CI$.

A $\DO$-module is said to be holonomic if its characteristic variety $\CHM$ has dimension $n=\dim\gO$. Then $\CHM$ is
homogeneous lagrangian and there exists a stratification $\gO=\bigcup \gO_\ga$ such
that $\CHM\subset\bigcup_\ga \overline{T^*_{\gO_\ga}\gO}$ \cite[Ch. 5]{KASHCOUR}.

Here a stratification of a manifold $\gO$ is a {\sl locally finite} union $\gO=\bigcup_\ga \gO_\ga$ such that
\begin{itemize}
\item For each $\ga$, $\overline{\gO}_\ga$ is an analytic subset
  of $\gO$ and $\gO_\ga$ is its regular part.
\item $\gO_\ga\cap \gO_\gb =\emptyset$ for $\ga\neq \gb$.
\item If $\overline{\gO}_\ga\cap \gO_\gb \neq \emptyset$ then
$\overline{\gO}_\ga\supset \gO_\gb$.
\end{itemize}
Let $Z$ be a submanifold of $\gO$ given in coordinates by $Z=\ensemble{(x,t)\in\gO}{t_1=\dots=t_p=0}$.
The polynomial $b$ is {\sl a $b$-function for $\CM$ along $Z$} if there exists in the ideal $\CI$ an
equation $b(\gth)+Q(x,t,D_x,D_t)$ where $\gth=t_1D_{t_1}+\dots+t_pD_{t_p}$ and $Q$ is of degree $-1$ for
the $V$-filtration. This means that $Q$ may be written as $\sum_it_iQ_i(x,t,D_x,[t_kD_{t_j}])$. This $b$-function is said to be {\sl tame}
if the roots of the polynomial $b$ are strictly greater than $-p$.

A more precise and intrinsic definition is given in
\cite{GL}  and \cite{RAMISCONF}, the definition is also extended to "quasi" or "weighted" $b$-functions" where $\gth$ is replaced by
$n_1t_1D_{t_1}+\dots+n_pt_pD_{t_p}$ for integers $(n_1\dots,n_p)$. In the definition of {\sl tame} the codimension $p$ of $Z$ is
replaced by $\sum n_i$. As this definition will not be explicitly used here, we refer to \cite{GL} for the details.

\begin{defin}\cite{GL}\label{def:quasitame}
The cyclic holonomic $\DO$-module $\CM$ is \textsl{tame} if there is a stratification
$\gO=\bigcup \gO_\ga$  such that $\CHM\subset \bigcup_\ga \overline{T^*_{\gO_\ga}\gO}$ and, for
each $\ga$, $\gO_\ga$ is open in $\gO$ or there is a tame quasi-$b$-function associated to $\gO_\ga$.
\end{defin}
The definition extends as follows:
\begin{defin}\cite{GL}\label{def:wtame}
The cyclic holonomic $\DO$-module $\CM$ is \textsl{weakly tame} if there is a stratification
$\gO=\bigcup \gO_\ga$  such that $\CHM\subset \bigcup_\ga \overline{T^*_{\gO_\ga}\gO}$ and, for
each $\ga$ one of the following is true:

(i) $\gO_\ga$ is open in $\gO$,

(ii) there is a tame quasi-$b$-function associated with $\gO_\ga$,

(iii) no fiber of the conormal bundle $T^*_{\gO_\ga}\gO$ is contained in $\CHM$.
\end{defin}

In (iii), the fibers of $T^*_{\gO_\ga}\gO$ are relative to the projection $\gp:T^*\gO\to\gO$. When $\gO_\ga$ is invariant under the action of a group compatible with the $\CaD$-module structure - which will be the case here, (iii) is equivalent to:

(iii)'  $T^*_{\gO_\ga}\gO$ is not contained in $\CHM$.

The following property of a weakly tame $\DO$-module has been proved in \cite{GL}:
\begin{thm}\label{thm:support}
If the holonomic $\DO$-module $\CM$ is weakly tame it has no quotient with support in a hypersurface of $\gO$.
\end{thm}

If $\gL$ is a real analytic manifold and $\gO$ its complexification, we also proved:
\begin{thm}\label{thm:lun}
Let $\CM$ be a holonomic weakly tame $\DO$-module, then $\CM$ has no distribution solution on $\gL$ with support in a hypersurface.
\end{thm}

We proved that under some additional conditions, the distribution solutions of a tame holonomic $\CaD$-module are locally
integerable that is in $L^1_{loc}$.

\subsection{$\CaD$-modules associated to the adjoint action.}\label{sec:adj}

Let $G$ be a complex reductive Lie group, $P$ a Lie subgroup, $\kg$ and $\kp$ their Lie algebras.

The differential of the adjoint action of $G$ on $\kg$ defines a  morphism of Lie algebra $\gt$ from $\kg$ to $\mathrm{Der}\CO[\kg]$ the Lie algebra of derivations on  $\CO[\kg]$ by:
\begin{equation}\label{formul1}
\left(\gt(Z)f\right)(X)=\frac d{dt}f\left(\exp(-tZ).X\right)|_{t=0} \quad\mathrm{for}\quad Z,X\in\kg, f\in \CO[\kg]
\end{equation}
i.e. $\gt(Z)$ is the vector field on $\kg$ whose value at $X\in\kg$ is $[X,Z]$. We denote by $\gt(\kg)$
the set of all vector fields $\gt(Z)$ for $Z\in\kg$. It is the set of vector fields on $\kg$ tangent to
the orbits of the adjoint action of $G$ on $\kg$. In the same way, $\gt(\kp)$ is the set of all vector
fields $\gt(Z)$ for $Z\in\kp$ and is the set of vector fields on $\kg$ tangent to the orbits of  $P$
acting on $\kg$.

The group $G$ acts on $\kg^*$, the dual of $\kg$. The space $\CO[\kg^*]$ of
polynomials on $\kg^*$ is identified with the symmetric algebra $S(\kg)$.
We denote by $\CO[\kg^*]^G=S(\kg)^G$ the space of invariant
polynomials on $\kg^*$ and by $\CO_+[\kg^*]^G=S_+(\kg)^G$ the subspace of polynomials vanishing at $\{0\}$. The
common roots of the polynomials in $\CO_+[\kg^*]^G$ are the nilpotent elements of $\kg^*$. 

 Let $\Dgg$ be the sheaf of differential operators on $\kg$ invariant under the adjoint action of $G$.
The principal symbol $\gs(R)$ of such an operator $R$ is a function on $T^*\kg=\kg\x\kg^*$ invariant
under the action of $G$.  If $F$ is a subsheaf of $\Dgg$, we denote by
$\gs(F)$ the sheaf of the principal symbols of all elements of $F$.

\begin{defin}\cite{RAMISCONF} \label{hctype}
A subsheaf $F$ of $\Dgg$ is of {\sl(H-C)-type} if $\gs(F)$ contains a power of $\CO_+[\kg^*]^G$ considered as
a subring of  $\CO_+[\kg\x\kg^*]^G$. A {\sl(H-C)-type $\Dg$-module} is the quotient $\MF$ of $\Dg$ by the
ideal $\CI_F$ generated by $\gt(\kg)$ and  by a subsheaf $F$ of (H-C)-type.
\end{defin}

As described in \cite[Examples 2.1.3. and 2.1.4]{RAMISCONF}, there are two main examples of (H-C)-type $\Dg$-module:

\begin{ex}
An element $A$ of $\kg$ defines a vector field with constant coefficients on $\kg$ by:
$$(A(D_x)f)(x)=\frac d{dt}f(x+tA)|_{t=0}\quad\mathrm{for}\quad f\in S(\kg^*), x\in\kg$$

By multiplication, this extends to an injective morphism from the symmetric algebra $S(\kg)$ to the algebra of
differential operators with constant coefficients on $\kg$; we will identify $S(\kg)$ with its
image and denote by $P(D_x)$ the image of $P\in S(\kg)$. If $F$ is a finite codimensional ideal of
$S(\kg)^G$, its graded ideal contains a power of $S_+(\kg)^G$ hence when it is identified to a set
of differential operators with constant coefficients, $F$ is a subsheaf of $\Dg$ of (H-C)-type and
$\MF$ is a $\Dg$-module of (H-C)-type.

If $\gl\in\kg^*$, the module $\MFl$ defined by Hotta and
Kashiwara \cite{HOTTA} is the special case where $F$ is the set of polynomials $Q-Q(\gl)$ for $Q\in
S(\kg)^G$.
\end{ex}

\begin{ex}\label{ex2}
The enveloping algebra $U(\kg)$ is the algebra of left invariant differential operators
on $G$. It is filtered by the order of operators and the associated graded algebra is
isomorphic by the symbol map to $S(\kg)$. This map is a $G$-map and defines a morphism
from the space of bi-invariant operators on $G$ to the space $S(\kg)^G$. This map is a
linear isomorphism, its inverse is given by a symmetrization morphism \cite[Theorem
3.3.4.]{VARADA}. Then, through the exponentional map a
bi-invariant operator $P$ defines a differential operator $\widetilde P$ on the Lie
algebra $\kg$ which is invariant under the adjoint action of $G$ (because the exponential
intertwines the adjoint action on the group and on the algebra) and the principal symbol
$\gs(\widetilde P)$ is equal to $\gs(P)$.

An eigendistribution $T$ is a distribution on an open subset of $G$ which is an eigenvector for all
bi-invariant operators $Q$ on $G$, that is satisfies $QT=\gl T$ for some $\gl$ in $\C$.

Let $U$ be an open subset of $\kg$ where the exponential is injective and $U_G=\exp(U)$. Let $T$ be
an invariant eigendistribution on $U_G$ and $\widetilde T$ the distribution on $U$ given by $\scal
T \gf = \scal{\widetilde T}{\gf_o\exp}$. As $T$ is invariant and eigenvalue of all bi-invariant
operators, $\widetilde T$ is solution of an (H-C)-type $\Dg$-module.
\end{ex}

In this paper, we fix a (H-C)-type subsheaf $F$ of $\Dgg$. We denote by $\MFG$  the
quotient of $\Dg$ by the ideal $\CI_F$ generated by
$\gt(\kg)$ and $F$. We denote by $\MFP$  the quotient of $\Dg$ by the ideal $\CJ_F$
generated by $\gt(\kp)$ and $F$. We have a canonical surjective morphism whose kernel will be denoted by $\CK_\kp$:
\begin{equation}\label{equ:morph}
\exacte{\CK_\kp}{\MFP}{\MFG}
\end{equation}

By example \ref{ex2}, the distribution of theorem \ref{thm:invariant} is solution of such  a module $\MFP$ (modulo transfer by
the exponential map).

The Killing form is a non-degenerate invariant bilinear form on the semi-simple Lie algebra $[\kg,\kg]$
satisfying $B([X,Z],Y)=B([X,Y],Z)$  We extend
it to a non-degenerate invariant bilinear form on $\kg$. This defines an isomorphism between $\kg$ and
its dual $\kg^*$.

The cotangent bundle to $\kg$ is equal to $\kg\x\kg^*$ identified to $\kg\x\kg$ by means of the Killing
form. Then it is known \cite[Prop 4.8.3.]{HOTTA} that if $\kN$ is the nilpotent cone of $\kg$, the
characteristic variety of $\MFG$ is equal to
\begin{equation}\label{equ:car0}
\ensemble{(X,Y)\in\kg\times\kg}{Y\in \kN, [X,Y]=0}
\end{equation}

In the same way:
\begin{lem}\label{lema}
The characteristic variety of $\MFP$ is contained in
\begin{equation}\label{equ:cark}
\ensemble{(X,Y)\in\kg\times\kg}{Y\in \kN, [X,Y]\in\kp^\bot}
\end{equation}
\end{lem}

\begin{proof}
Let us first consider that variety as a subset of  $\kg\x\kg^*$. The characteristic variety of $\MFP$ is
contained in the variety defined by $F$ that is the nilpotent cone of $\kg^*$. On the other hand, it is
contained in the variety defined by $\gt(\kp)$ that is $$\ensemble{(X,\gx)\in\kg\times\kg^*}{\forall
Z\in\kp \quad\scal{[X,Z]}{\gx}=0}$$

 The isomorphism defined by the Killing form exchanges the nilpotent cone of $\kg$ and that of $\kg^*$, hence
 after this isomorphism the characteristic variety is a subset of $\kg\x\kg$ contained in
$$\ensemble{(X,Y)\in\kg\times\kg}{Y\in \kN, \forall Z\in\kp \quad B([X,Z],Y)=0}$$
But we have $B([X,Z],Y)=B([X,Y],Z)$ which gives the result.
\end{proof}

\begin{rem}
Using theorem \ref{thm:reg}, it is not difficult to show that the characteristic variety of $\MFP$ is in fact
\textsl{equal} to the set (\ref{equ:cark}).
\end{rem}

The variety (\ref{equ:car0}) is lagrangian \cite{HOTTA} hence the module $\MFG$ is always holonomic but
in general the variety (\ref{equ:cark}) is not lagrangian and $\MFP$ is not holonomic. We will see that
it is the case when $G=\Gln(\C)$ and $P$ is the set of matrices fixing a non zero vector in $\C^n$, or $G=\Sln(\C)$
and $P$ a maximal parabolic group.

\subsection{Main Result}

To state the main results, we restrict 
to the following cases:
\begin{itemize}
\item $G$ is the group $\Gln(\C)$ acting on $\C^n$ by the usual action and $P$ is the stability subgroup
of $G$ at $v_0\in\C^n$, that is $P=\ensemble{g\in G}{g.v_0=v_0}$.
\item $G$ is the group $Sl_n(\C)$ acting on the projective space $\P_{n-1}(\C)$ and $P$ is a
maximal parabolic subgroup, that is the stability group of a point in  $\P_{n-1}(\C)$.
\item $G$ is a product of several groups $\Gln(\C)$ and  $Sl_n(\C)$ and $P$ is the
corresponding stability group.
\end{itemize}

In the first two cases, all subgroups $P$ are conjugated (except the trivial case $v_0=0$). The third case
will be useful during the proof.

Let $\kg$ and $\kp$ be the Lie algebras of $G$ and $P$. Our main result which will be proved in \S
\ref{sec:proofmain} is the following:

\begin{thm}\label{thm:main}
For any subsheaf $F$ of $\Dgg$ of (H-C)-type, the $\Dg$-module $\MFP$ is holonomic and weakly tame.
\end{thm}

Let $\kg_\R$ be a real form of $\kg$ and $G_\R$ be the corresponding group. Theorem \ref{thm:main} and
theorem 1.5.7. in \cite{GL} implies:

We say that a distribution is singular if it is supported by a hypersurface. Then:

\begin{cor}\label{cor:main}
 $\MFP$ has no singular distribution (or hyperfunction) solution on an open set of $\kg_\R$.
\end{cor}

In \cite[corollary 1.6.3]{GL} we proved that $\MFG$ has a stronger property: all its
solutions are $L^1_{loc}$. Here we prove only that $\MFP$
is weakly tame but we will still be able to show that all solutions are  $L^1_{loc}$.

Let $\kg_\R$ be $\gln(\R)$, then $G_\R$ is equal to $\Gln(\R)$. Let $v_0$ be a non zero vector of $\R^n$
and $P_\R$ be the stability group of $v_0$, $\kp_\R$ its Lie algebra. Remark that $\kp_\R$ is a real
form for $\kp$. The same is true if $\kg_\R$ is $\sln(\R)$, $\sln(\C)$ or $\gln(\C)$ viewed as a real form of
$\kg\kl_{2n}(\C)$.

\begin{thm}\label{thm:main2}
Assume that $\kg_\R$ is $\gln(\R)$, $\gln(\C)$, $\sln(\R)$, $\sln(\C)$ or any direct sum of them. Let
$G_\R$ be the corresponding Lie group and $P_\R$ the stability group of a real point.

Then any distribution solution of  $\MFP$  which is invariant under the action of $P_\R$ is a
$L^1_{loc}$ function invariant under the action of $G_\R$.
\end{thm}

\begin{proof}[Proof of theorem \ref{thm:invariant}]
From example \ref{ex2}, we know that a distribution satisfying the conditions of theorem \ref{thm:invariant}  is a solution of
a module  $\MFP$  hence is $G$-invariant. So as in Baruch \cite{BARUCH}, this theorem is an easy consequence of theorem \ref{thm:main2}.
\end{proof}

However if $\kg_\R$ is a real form of $\Gln(\C)$ different from $\Gln(\R)$ or $\Gln(\C)$, the
intersection of $\kp$ and $\kg_\R$ is not a real form for $\kp$. Then corollary \ref{cor:main} is still
true but the solutions of  $\MFP$  do not correspond to the action of a group.

\section{Example: the $\sld$-case}\label{sec:sld}

We consider the canonical base of $\sld$:

\begin{equation*}
H=\begin{pmatrix}
1 & 0 \\
0 & -1
\end{pmatrix}\qquad
X=\begin{pmatrix}
0 & 1 \\
0 & 0
\end{pmatrix}\qquad
Y=\begin{pmatrix}
0 & 0 \\
1 & 0
\end{pmatrix}
\end{equation*}
and the general matrix of  $\sld$ is written as $Z=xH+yX+zY$. Let $\kg=\sld$ and define $\kp$ as the
subspace generated by $H$ and $X$.

In coordinates $(x,y,z)$ we have:
\begin{align*}
\gt(H)&=2(zD_z-yD_y)\\
\gt(X)&=-zD_x+2xD_y\\
\gt(Y)&=yD_x-2xD_z
\end{align*}

By definition, the value of $x\gt(H)+y\gt(X)+z\gt(Y)$ at the point $Z=xH+yX+zY$
is $[xH+yX+zY,xH+yX+zY]=0$ hence we have $x\gt(H)+y\gt(X)+z\gt(Y)=0$.

Here $\gt(\kp)$ is generated by $(\gt(H),\gt(X))$ while $\gt(\kg)$
is generated by $(\gt(H),\gt(X),\gt(Y))$ hence the kernel of $\MFP\to\MFG$
is the submodule of $\MFP$ generated by $\gt(Y)$. This defines an exact sequence:
\begin{equation*}
0\xrightarrow{\ \ }{\CK_F}\xrightarrow{\ \gt(Y)\ }{\MFP} \xrightarrow{\ \ \ \ \ }{\MFG} \xrightarrow{\ \ }0
\end{equation*}
The module $\CK_F$ is the quotient of $\Dg$ by the ideal $\CJ=\ensemble{Q\in\Dg}{Q\,\gt(Y)\in\CJ_F}$.
Here $\CJ_F$ is the ideal of $\Dg$ generated by $\gt(\kp)$ and a subsheaf $F$
of $\Dgg$ is of (H-C)-type.

The equations
\begin{align*}
z\gt(Y)&=-x\gt(H)-y\gt(X)\qquad [\gt(X),\gt(Y)]=\gt(H) \qquad[\gt(H),\gt(Y)]=-2\gt(Y)\\
\textrm{ and } & [Q,\gt(Y)]=0 \textrm{ for any } Q\in F\subset\Dgg
\end{align*}
show that the $z$,$\gt(X)$, $\gt(H)+2$ and $F$ are contained in $\CJ$.

The characteristic variety of $\CK_F$ is thus contained in the set defined by $z$ and $F$, that is in
$\ensemble{(x,y,z,\gx,\gh,\gz)\in \kg\times\kg^*}{z=0, \gx^2+4\gh\gz=0}$. This variety being involutive, its
ideal of definition is stable under Poisson bracket and
we have also the equation $\{z,\gx^2+4\gh\gz\}=-4\gh$ so the characteristic variety of $\CK_F$
is $$\ensemble{(x,y,z,\gx,\gh,\gz)\in \kg\times\kg^*}{z=0, \gx=0, \gh=0}$$
 that is the conormal bundle to
$S=\{z=0\}$. This implies that $\CK_F$ is isomorphic to a power of $\CB_{S|\kg}=\Dg/\Dg z+\Dg D_x+\Dg D_y$.

For example, if $F=\{D_x^2+4D_yD_z-\gl\}$, then $\CJ$ is generated by $(z,D_y,D_x^2)$.

Consider now distribution solutions of these modules, they make an exact sequence:
\begin{equation*}
0\xrightarrow{\ \ }{Sol(\MG)}\xrightarrow{\ \ \ \ \ }{Sol(\MP)} \xrightarrow{\ \gt(Y)\ }{Sol(\CK_F)}
\end{equation*}
A solution of $\CK_F$ is canceled by $z$ and solution of a system isomorphic to a power of $\CB_{S|\kg}$,
hence it is of the form $\gf(x,y)\gd(z)$ where $\gf(x,y)$ is analytic and $\gd(z)$ is the Dirac
distribution.

\begin{prop}
The module $\MP$ has no singular distribution solution.
\end{prop}

\begin{proof}
Let $T$ be a  singular distribution solution of $\MP$. Outside of $\{z=0\}$, a solution of $\MP$
satisfies $\gt(Y)T=0$ hence is a solution of $\MG$. From \cite[cor 1.6.3]{GL} this implies that $T$ vanishes
outside of$\{z=0\}$ hence is of the form $T=T_1(x,y)\gd(z)$.

We must have $\gt(Y)T(x,y,z)=\gf(x,y)\gd(z)$ hence
$yD_xT_1(x,y)\gd(z)-2xT_1(x,y)\gd'(z)=\gf(x,y)\gd(z)$. So
$T_1(x,y)$ satisfy $yD_xT_1(x,y)=\gf(x,y)$ and $xT_1(x,y)=0$. As $\gf$ is analytic this implies that $\gf=0$.

So $T(x,y,z)=\gd(x)\gd(y)\gd(z)$, but
then we would have $\gt(Y)T(x,y,z)=0$ and thus $T$ would be a singular solution of $\MG$ which is impossible.
\end{proof}

\section{General results on inverse image by invariant maps.}

In the section, we will prove some general results on the $\CaD$-module associated to an action of a
group $G$ on a manifold.

\subsection{Inverse image of a $\CaD$-module.}\label{sec:inverse}

We begin with elementary properties of inverse images that can be find for example in \cite{BJORK}.

Let $\gF:U\to V$ be a holomorphic map between two complex analytic manifolds. The inverse image of a
coherent $\DV$-module $\CM$ by $\gF$ is, by definition, the
$\DU$-module:
\begin{equation*}
\gF^*\CM=\OU\ox_{\gF^{-1}\OV}\gF^{-1}\CM
\end{equation*}
The module $\gF^*\CM$ is not always coherent but this is the case if $\CM$ is
holonomic or if $\gF$ is a submersion.

When $\gF$ is the canonical projection $U\x V\to V$, the module $\gF^*\CM$
is the external product $\OU\widehat\ox\CM$ hence if $\CM=\DV/\CI$ where $\CI$
is a coherent ideal of $\DV$ then $\gF^*\CM=\CaD_{U\x V}/\CJ$ where $\CJ$ is
the ideal of $\CaD_{U\x V}$ generated by $\CI$.

Suppose now that $\gF:U\to V$ is a submersion and let $\CI$ be a coherent ideal of $\DV$. We consider
the subset $\CJ_0$ of $\DU$ defined in the following way:

An operator $Q$ defined on an open subset $U'$ of $U$ is in $\CJ_0$ if and only if there exits
some differential operator $Q'$ on $\gF(U')$ belonging to $\CI$ and such that for any holomorphic
function $f$ on $V$ we have $Q(f_\circ\gF)=Q'(f)\circ\gF$.

Then $\gF^*\CM=\DU/\CJ$ where $\CJ$ is the ideal of $\DU$ generated by $\CJ_0$. The problem being
local on $U$, this is easily deduced from the projection case.

Let $G$ be a group acting on a manifold $U$. To an element $Z$ of the Lie
algebra $\kg$ of $G$ we associate a vector field $\gt_U(Z)$ on $U$ defined as in (\ref{formul1}) by:
\begin{equation}\label{formul2}
\gt_U(Z)(f)(x)=\frac{d}{dt}f(exp(-tZ).x)|_{t=0}
\end{equation}

\begin{lem}\label{lemi}
Let $\gF:U\to V$ be a submersive map of $G$-manifolds satisfying
$\gF(g.x)=g.\gF(x)$ for any $(g,x)\in G\x U$.
Let $\CI$ be a coherent ideal of $\DV$ and $\CM$ be the coherent
$\DV$-module $\DV/\CI$. Then the inverse image $\gF^*\CM$ of $\CM$ by $\gF$ is a coherent $\DU$-module $\DU/\CJ$ such that:

For any $Z\in\kg$, $\gt_V(Z)$ belongs to $\CI$ if and only if $\gt_U(Z)$
belongs to $\CJ$.
\end{lem}

\begin{proof}
An direct calculation shows that $\gt_V(Z)(f_\circ \gF)=\gt_U(Z)(f)\circ \gF$
which shows immediately the lemma.
\end{proof}

\subsection{Equivalence.}\label{sec:equiv}
Let $G$ be a complex Lie group acting transitively on a complex manifold $\gO$.
Let $v_0\in\gO$ and let $P=G^{v_0}$ be the stability subgroup at $v_0$, hence
$\gO$ is isomorphic to the quotient $G/P$.

We denote by $(g,v)\mapsto g\dt v$ the action of $G$ on $\gO$ and by
$(g,X)\mapsto g\dt X$ the adjoint action of $G$ on its Lie algebra $\kg$. Then
$G$ acts on $\kg\x\gO$ by $g\dt(X,v)=(g\dt X,g\dt v)$. The group $P$ acts on
$\kg$ by restriction of the action of $G$.

Let $U$ be an open subset of $\gO$ containing $v_0$ and $\gf$ a
holomorphic map $\gf:U\to G$ such that $\gf(v)\dt v_0=v$ for all $v$ in $U$.

This defines a submersive morphism  $\gF:\kg\x U\to \kg$ by
$\gF(X,v)=\gf(v)^{-1}\dt X$. The subsets of $\kg\x U$ invariant under $G$ are
exactly the sets $\gF^{-1}(S)$ where $S$ is an orbit of $P$ on $\kg$.

{\bf Remark:}
It is known that $\gF$ defines an equivalence between distributions on $\kg\x U$
invariant under $G$ and distributions on $\kg$ invariant under $P$ (see Baruch
\cite{BARUCH} for example). We will prove a similar result for $\CaD$-modules.
However, in the case of distributions the map $\gF$ is of class $\mathcal C^\infty$
hence may be globally defined. Here we need a holomorphic map and such a section is not
defined globally on an open set $U$ stable under $G$. This is of no harm as long as we
consider locally the vector fields tangent to the orbits. In this section, when we
speak of $G$-orbits on $\kg\x U$, it means the intersection of $\kg\x U$ with a $G$-orbit of $\kg\x \gO$.

For $X\in\kg$ the action of $G$ on $\kg\x\gO$ and on $\kg$ defines vector fields
$\gt_{\kg\x\gO}(X)$ on $\kg\x\gO$ and $\gt_\kg(X)$ on $\kg$ through formula
(\ref{formul2}).

Let $\kp$ be the Lie algebra of $P$ and denote by $\gt(\kp)$ the set of vector
fields $\gt_\kg(X)$ for $X\in\kp$. Let us denote by $\gt_*(\kg)$  the set of
vector fields $\gt_{\kg\x\gO}(X)$ for $X\in\kg$. Define now $\CN_{\gt_*(\kg)}$ as the
quotient of $\CaD_{\kg\x\gO}$ by the ideal generated by $\gt_*(\kg)$ and
$\CM_{\gt(\kp)}$ as the quotient of $\CaD_\kg$ by the ideal generated by $\gt(\kp)$.

\begin{lem}\label{lem:iso1}
The map $\gF$ defines an isomorphism between the restrictions to $\kg\x U$ of
the  $\CaD_{\kg\x\gO}$-modules $\CN_{\gt_*(\kg)}$ and $\gF^*\CM_{\gt(\kp)}$.
\end{lem}

\begin{proof}
Let $\Psi$ be the map $\kg\x U\to \kg\x U$ given by $\gY(X,v)=(\gF(X,v),v)$. It
is an isomorphism which exchanges the $G$-orbits on $\kg\x U$ with the product
by $U$ of the $P$-orbits on $\kg$. Hence it exchanges the vector fields tangent
to the $G$-orbits that is $\gt_*(\kg)$ with the product of the set $\gt(\kp)$ of
vector fields on $\kg$ tangent to the $P$-orbits by the set $\CT_U$ of all
vector fields on $U$ that is $\gt(\kp)\widehat\ox \CT_U$.

The quotient of $\CaD_{\kg\x U}$ by  $\gt(\kp)\widehat\ox \CT_U$ is precisely
$p^*\CM_\gt$ where $p:\kg\x U\to \kg$ is the canonical projection $p(X,v)=X$ (see the previous section).
As $\gF=p\circ\gY$, we are done.
\end{proof}

We may also define the module $\CM_{\gt(\kg)}$ as the quotient of $\CaD_\kg$ by
 the ideal generated by $\gt(\kg)$. Then we have:

\begin{lem}\label{lem:iso2}
The map $\gF$ defines an isomorphism between the restrictions to $\kg\x U$ of
the  $\CaD_{\kg\x\gO}$-modules $\CM_{\gt(\kg)}\widehat\ox\CO_U$ and $\gF^*\CM_{\gt(\kg)}$.
\end{lem}

\begin{proof}
The inverse image by $\gF$ of a $G$-orbit is the product of that $G$-orbit by $U$ hence
the proof is the same than the proof of lemma (\ref{lem:iso1}).
\end{proof}

Let $Q$ be a differential operator on $\kg$, then $Q\ox 1$ is a differential operator on
$\kg\x U$ and as $\gY$ is an isomorphism, this defines $\gY^*(Q\ox 1)$ as a differential
operator on $\kg\x U$. If $Q$ is $P$-invariant, then $\gY^*(Q\ox 1)$ is $G$-invariant on
$\kg\x U$ and if $Q$ is $G$-invariant on $\kg$ then $\gY^*(Q\ox 1)$ is equal to $Q\ox 1$. We
denote $\widetilde\gY(Q)=\gY^*(Q\ox 1)$.

Let $F$ be a set of differential operators on $\kg$ invariant under the
$P$-action, we consider four $\CaD$-modules:
\begin{itemize}
\item $\CM_{F,\kp}$ is the quotient of $\CaD_\kg$ by the ideal
generated by $F$ and $\gt_\kg(\kp)$
\item $\CM_{F,\kg}$ is the quotient of $\CaD_\kg$ by the ideal
generated by $F$ and $\gt_\kg(\kg)$
\item $\CN_{F,\kg}$ is the quotient of $\CaD_{\kg\x\gO}$ by the ideal
generated by $\widetilde\gY(F)$ and $\gt_*(\kg)=\gt_{\kg\x\gO}(\kg)$
\item the product $\CM_{F,\kg}\widehat\ox\CO_\gO$
\end{itemize}

As a consequence of Lemma \ref{lem:iso1} and Lemma \ref{lem:iso2} we have the following result:

\begin{prop}\label{prop:iso3}
The $\CaD_{\kg\x U}$-modules $\CN_{F,\kg}$ and $\gF^*(\CM_{F,\kp})$ are
isomorphic as well as  $\CM_{F,\kg}\widehat\ox\CO_\gO$ and  $\gF^*(\CM_{F,\kg})$.

These isomorphism are
compatible with the morphisms $\CM_{F,\kp}\to\CM_{F,\kg}$ and $\CN_{F,\kg}\to \CM_{F,\kg}\widehat\ox\CO_\gO$.

If the operators of $F$ are $P$-invariant the operators of $\widetilde\gY(F)$
are $G$-invariant and if they are $G$-invariant then those of $\widetilde\gY(F)$
are $G$-invariant and independent of $v\in U$.
\end{prop}

\subsection{Reduction to a subalgebra}\label{sec:reduc}

We assume now that $G$ is a reductive Lie group operating on a manifold $\gO$
hence on $\kg\x\gO$. The algebra $\kg$ is reductive hence $[\kg,\kg]$ is a
semi-simple Lie algebra with a non-degenerate Killing form $B$. We extend the
form $B$ to a non-degenerate invariant bilinear form on $\kg$ that we still
denote by $B$.

Let $S\in \kg$ be a semi-simple element and $\km=\kg^S$, the reductive Lie
subalgebra of elements commuting with $S$. Let $\kq=\km^\bot$ the orthogonal for
the form $B$ and
$\km''=\ensemble{Y\in\km}{\textrm{det}(ad Y)|_\kq\ne 0}$, let $M=G^S$ the
associated Lie group.

We consider the map $\gY:G\x\km''\x \gO\to\kg \x \gO$ defined by
$\gY(g,Y,v)=(g\dt Y,g\dt v)$. As $\kg=\km\oplus [\kg,S]$, $\gY$ is a submersion
onto the open set $G\km''\x \gO$. If $U$ is a $G$-invariant open subset of
$\kg\x\gO$,
$\gY^{-1}(U)$ is equal to $G\x U'$ for some open subset $U'$ of $\km''\x \gO$
invariant under the action of $M$.

Let $F$ be a (H-C)-type subsheaf of $\Dgg$ defined on $U$. According to
definition (\ref{hctype}), F is a subsheaf of $\Dgg$ such that $\gs(F)$ contains a power
of $\CO_+[\kg^*]^G$. Let $\gt_*(\kg)$ be the sheaf of vector fields tangent to the
orbits of $G$ on $\kg\x \gO$ as in section \ref{sec:equiv}.

Let $\CN_{F,\kg}$ be the coherent $\Dgo$-module defined on $U$ as
the quotient of $\Dgo$ by the ideal generated by $F$ and $\gt_*(\kg)$.

Remark that here we assume that the operators of $F$ are $G$-invariant.
For such operators $Q$ we have $\widetilde\gY(Q)= Q\ox 1$ hence we may confuse
$\widetilde\gY(F)$ and $F$.

\begin{thm}\label{thm:reg}
There exists a (H-C)-type subsheaf $F'$ of ${\CaD_\km^M}$ on $U'$ such that
$\gY^*\CN_{F,\kg}\simeq\CO_G\hat\ox\CN_{F'\!\!,\,\km}$
on $\gO$.
\end{thm}
\begin{proof}
The map $\gY$ is a submersion hence $\gY^*\CN_{F,\kg}$ is coherent and canonically a
quotient of
$\CaD_{G\x\km''\x \gO}$ by an ideal $\CJ$.

Consider the action of $G$ on $G\x\km''\x \gO$ given by $g'.(g,A,v)=(g'g,A,v)$.
The map $\gY$ is compatible with this action of $G$ hence we may apply lemma \ref{lemi}
to the inverse image $\gY^*\CN_{F,\kg}$. We get that $\CJ$ is an ideal containing the
vector fields $\gt_G(X)$ for all $X\in\kg$ that is all vector fields on $G$. This shows that
$\CJ$ is the product of $\CaD_G$ by an ideal of $\CaD_{U'}$.
Hence $\gY^*\CN_{F,\kg}=\CO_G\hat\otimes \CN$ where $\CN$ is some holonomic module on
$U'$.

Consider now the action of $M$ on $G\x\km''\x \gO$ given by
$$m.(g,A,v)=(mgm^{-1},m\dt A, m\dt v)$$
and on $\kg\x \gO$ induced by that of $G$.
We may again apply lemma \ref{lemi}. We get that $\CN$ is equal to the quotient
of $\CaD_{\km\x \gO}$ by an ideal $\CI$ which contains the vector fields
$\gt_{\km\x \gO}(X)$
for any $X\in\km$.

We will now define the set $F'$ from $F$. As $S$ is semi-simple we have
$\kg=\km\oplus[\kg,S]$ hence a local isomorphism $\gy:[\kg,S]\ox\km''\ox\kg$
given by $\gy(X,m)=exp(X).m$. In coordinates $(x,t)$ induced by this
isomorphism, all derivations in $x$ are in the ideal generated by the vector
fields tangent to the $G$-orbits.

After division by these derivations an operator $Q$ invariant under $G$ depends
only on $(t,D_t)$ i.e. is a differential operator on $\km$ invariant under the
action of $M$. Denote by $\gy^*Q$ this operator. If the principal symbol of $Q$
is a function of  $\CO[\kg^*]^G$, the principal symbol of $\gy^*Q$ is its
restriction to $\CO[\km^*]^M$. Hence if $F$ is an (H-C)-type subsheaf of $\Dgg$,
$F'=\gy^*F$ is an (H-C)-type subsheaf of ${\CaD_\km^M}$. Then the ideal $\CI$ is
generated by $F'$ and $\gt_{\km\x V}(\km)$ which shows the theorem.
\end{proof}

\section{The $\Gln(\C)$ and $\Sln(\C)$ cases}\label{sec:proofmain}

\subsection{Main proof}

WAssume now that $G$ is the linear group $\Gln(\C)$
acting on $V=\C^n$ by the standard action. Then $P$ is the subgroup of matrices which leave
invariant a point $v_0\in V =\C^n$ and its Lie algebra $\kp$ is the set of matrices which cancel
$v_0$. If $v_0=0$ $P=G$ and everything is trivial otherwise $v_0\in V^*=\C^n-\{0\}$ and all
subgroups $P$ are conjugate.

It is known \cite{WAL} that a $G$-orbit in $\kg$ splits into a finite number of $P$-orbits. More precisely,
let $\kg^{(d)}$ be the set of matrices $A$ such that the vector space generated by $(A^pv_0)_{p=0,\dots,n-1}$
is $d$-dimensional. Then the $P$-orbits are exactly the intersections of the $G$-orbits with the
varieties $\kg^{(d)}$. In particular, $\kg^{(n-1)}$ is a Zarisky open subset of $\kg$ where $P$-orbits and $G$-orbits coincide.

\begin{rem}\label{rem:sigma}.
Let $\gS$ be the complementary of $\kg^{(n-1)}$. It is a hypersurface of $\kg$. Outside of $\gS$,
$P$- and $G$-orbits coincide, hence the vector fields $\gt(\kp)$ and $\gt(\kg)$ are the same. So
the kernel $\CK_\kp$ of $\MFP\to\MFG$ is supported by $\gS$.
\end{rem}

More generally,  we will consider a product
\begin{equation}
G=\prod_{k=1}^N Gl_{n_k}(\C)\qquad\textrm{acting on }\qquad V=\prod_{k=1}^N \C^{n_k}
\end{equation}

Let $F$ be a (H-C)-type subset of $\Dgg$,
 we may consider the $\CaD$-modules $\MFG$, $\MFP$ and $\NFG$ as in section \ref{sec:equiv}. We will show:

\begin{prop}\label{prop:strat}
There is a stratification $\kg=\bigcup \kg_\ga$ such that

(1) The characteristic variety of $\MFP$  is contained in the union of the conormals to the strata  $\kg_\ga$

(2) For each $\ga$, if the conormal to $\kg_\ga$ is contained in the characteristic variety of $\MFP$,
then $\MFP$ admits a tame quasi-$b$-function along $\kg_\ga$.
\end{prop}

By definition this shows that the module $\MFP$ is holonomic and weakly tame (theorem \ref{thm:main}).
In the proof we will encounter three situations:

a) the conormal to $\kg_\ga$ is not contained in the characteristic variety of $\MFP$

b) the module $\MFP$ is isomorphic to $\MFG$ in a neighborhood of $X_\ga$ which implies the existence of a tame $b$-function because $\MFG$ is tame.

c) the module $\MFP$ is a power of the module associated to a normal crossing divisor and is trivially tame.

Remark that we will never need to explicit the definition of a tame $b$-function here. We will get it from results of \cite{GL} concerning the module $\MFG$.

By proposition \ref{prop:iso3}, proposition \ref{prop:strat} is equivalent to the following:

\begin{prop}\label{prop:}
There is a stratification $\kg\x V=\bigcup X_\ga$ such that

(1) The characteristic variety of $\NFG$  is contained in the union of the conormals to the strata  $X_\ga$

(2) For each $\ga$, if the conormal to $X_\ga$ is contained in the characteristic variety of $\NFG$, then $\NFG$
admits a tame quasi-$b$-function along $X_\ga$.
\end{prop}

\subsection{Stratification}
Let us first recall the stratification that we defined in \cite{GL} on any semi-simple algebra $\kg$.

Fix a Cartan subalgebra $\kh$ of $\kg$ and denote by $\gD=\gD(\kg,\kh)$ the root system associated to
$\kh$. For each $\ga\in\gD$ we denote by $\kg_\ga$ the root subspace corresponding to $\ga$ and by
$\kh_\ga$ the subset $[\kg_\ga,\kg_{-\ga}]$ of $\kh$.

Let $\CF$ be the set of the subsets $\gp$ of $\gD$ which are closed and
symmetric that is such that $(\gp+\gp)\cap\gD\subset \gp$ and $\gp=-\gp$. For
each $\gp\in\CF$ we define $\khp=\sum_{\ga\in \gp}\kh_\ga$,
$\kgp=\sum_{\ga\in \gp}\kg_\ga$, $\khpo=\ensemble{H\in\kh}{\ga(H)=0
\textrm{ if } \ga\in \gp}$, $\khpp =\ensemble{H\in\khpo}{\ga(H)\ne0 \textrm{ if } \ga\notin \gp}$
and $\kqp=\khp+\kgp$. $\kqp$  is a semisimple Lie subalgebra of $\kg$

\begin{rem}
With the notations of \S \ref{sec:reduc} we have
$\km=\kh\oplus\kg_P$ and $\km''=(\kh_P^\bot)'\oplus\kh_P\oplus\kg_P$.
\end{rem}

To each $\gp\in\CF$ and each nilpotent orbit $\kO$ of $\kqp$ we
associate a conic subset of $\kg$
\begin{equation*}
S_{(\gp,\kO)}=\bigcup_{x\in\khpp}G.(x+\kO)\label{def:strat}
\end{equation*}
It is proved in \cite{GL} that these sets define a finite stratification of $\kg$ independent of the choice of $\kh$.

If $\kg$ is a reductive Lie algebra, we get a stratification of $\kg$ by adding the center $\kc$ of $\kg$ to any
stratum of the semi-simple algebra $[\kg,\kg]$:
\begin{equation*}
\widetilde S_{(\gp,\kO)}=S_{(\gp,\kO)}\oplus \kc
\end{equation*}

This applies in particular to $\gln(\C)$. For a matrix $X$ of $\gln(\C)$ and a vector $v$ of $\C^n$, we
denote by $d(X,v)$ the dimension of the vector space generated by $(v,Xv,X^2v,\dots,X^{n-1}v)$ where
$Xv$ denotes the usual action. If $X=X_1+\dots+X_q$ is an element of $\oplus \kg\kl_{n_i}(\C)$,  $d(X,v)$ is the sum
$\sum  d(X_i,v_i)$.

Let $v_0$ be a non-zero vector of $\C^n$. To each $\gp\in\CF$, each nilpotent orbit $\kO$ of $\kqp$ and
each integer $p\subset [0\dots n-1]$ we associate:
\begin{equation*}
S_{(\gp,\kO,p)}=\ensemble{X\in \widetilde S_{(\gp,\kO)}}{d(X,v_0)=p}   \label{def:strat2}
\end{equation*}
The sets $\ensemble{X\in \kg}{d(X,v_0)=p}$ form a finite family of closed algebraic subsets of
$\kg$ hence the sets $S_{(\gp,\kO,p)}$ define a new stratification of $\kg$.

In the same way, we define a stratification of $\kg\x V$ by
\begin{equation*}
T_{(\gp,\kO,p)}=\ensemble{(X,v)\in \kg\x V}{X\in\widetilde S_{(\gp,\kO)},\ d(X,v)=p}   \label{def:strat3}
\end{equation*}
If $\gF$ is the map $\gF:\kg\x U\to \kg$ defined by a map $\gf:U\to G$ as in section \ref{sec:equiv}, we
have $\gF^{-1} (S_{(\gp,\kO,p)}=T_{(\gp,\kO,p)})$.

The stratification $\left(\widetilde S_{(\gp,\kO)}\right)$ has been associated to $\MFG$ in \cite{GL}.
We will associate $\left(S_{(\gp,\kO,p)}\right)$ to $\MFP$ and $\left(T_{(\gp,\kO,p)}\right)$ to $\NFG$.

To end this section let us calculate the characteristic variety of the module $\NFG$ when $G=\Gln(\C)$. On
$\kg=\kg\kl_n(\C)$ we consider the scalar product $(A,B)\mapsto \textrm{trace}(AB)$ which extends the
Killing form of $\ks\kl_n(\C)$. This identifies $\kg$ and $\kg^*$ and in the same way the usual
hermitian product $(u,v)\mapsto\scal u{\bar v}$ on $\C^n$ identifies $V$ and $V^*$.

 If $u$ and $v$ are two vectors of $V=\C^n$ we denote by $u\wedge \bar v$ the $(n,n)$-matrix whose entry $(i,j)$ is $u_i\bar v_j$.

\begin{prop}\label{prop:car2}
The characteristic variety of $\NFG$ is contained in
\begin{equation}
\ensemble{(X,u,Y,v)\in\kg\x V\x \kg\x V}{Y\in \kN, [X,Y]=u\wedge \bar v}
\end{equation}
\end{prop}

\begin{proof}
The proof is similar to the proof of lemma \ref{lema}. The vector field $\gt_{\kg\x V}(Z)$
has value $([X,Z],Zu)$ at the point $(X,u)\in\kg\x V$ hence the characteristic variety of  $\NFG$
is contained in the set of points $(X,u,Y,v)\in\kg\x V\x \kg\x V$ satisfying
$B([X,Z],Y)+\scal {Zu}{\bar v}=0$ for any $Z\in\kg$.

But we have $\scal{Zu}{\bar v}=\sum Z_{ij}u_j\bar v_i=B(Z,u\wedge\bar v)$ hence
$B(Z,[X,Y]-u\wedge\bar v)=0$ for any $Z$ which means that  $[X,Y]=u\wedge\bar v$.
\end{proof}

\subsection{Nilpotent points}\label{sec:nilpot}

In this section, we take $G=\Gln(\C)$, $\kg=\gln(\C)$, $v_0$ is a non zero vector of $\C^n$, $P=G^{v_0}$
and $\kp$ its Lie algebra.

\begin{lem} \label{lemb}
Let $X$ be a regular nilpotent element of $\kg$. Then if the orbit $P.X$
is not open dense in the orbit $G.X$, there exists a semi-simple element $Y$ in $\kg$
which is not in the center of $\kg$ and such that $[X,Y]\in \kp^\bot$.
\end{lem}

\begin{proof}
Let $\kg$ act on the vector space $V=\C^n$ by $(X,v)\in \kg\x V\to Xv$.
If $X$ is nilpotent regular, its Jordan form has only one block, we deduce easily the following statements:
\begin{itemize}
\item the kernel $H$ of $X^{n-1}$ is a hypersurface
\item the image of $V$ by $X$ is $H$
\item if $v\notin H$, $(v,Xv,X^2v,\dots,X^{n-1}v)$ is a basis of $V$
\end{itemize}

So, there is a unique integer  $p\in[0,\dots,n-1]$ and some $w\notin H$ such that
$v_0=X^pw$. Then $(w,Xw,\dots,v_0=X^pw,Xv_0,\dots,X^{n-p-1}v_0)$ is a basis of $V$.

If $X$ and $X'$ are two regular nilpotent matrices with the same characteristic integer $p$, the matrix
of $\Gln(\C)$ which sends $(w,Xw,\dots,v_0=X^pw,Xv_0,\dots,X^{n-p-1}v_0)$ on $(w',X'w',
\dots,v_0={X'}^pw',X'v_0,\dots,{X'}^{n-p-1}v_0)$ sends $v_0$ on itself hence is an element of $P$ which
conjugates $X$ and $X'$.

The $P$-orbits in the $G$-orbit of nilpotent regular matrices are thus given by this integer $p$. We
have $p=0$ if and only if $v_0\notin H$ hence the $P$-orbit given by $p=0$ is open in the $G$-orbit,
that is the first alternative of the lemma.

Consider now the case $p\ge 1$. Let $V_1$ be the span of $(w,Xw,\dots,X^{p-1}w)$ and $V_2$ be the span
of $(v_0,Xv_0,\dots,X^{n-p-1}v_0)$. We have $V=V_1\oplus V_2$, $XV_1\subset V_1\oplus\C v_0$ and
$XV_2\subset V_2$.

Let $(a,b)\in\C^2$, $a\neq b$ and $\gF_{ab}=aI_{V_1}+bI_{V_2}$. ($I_{V_i}$ is the identity morphism on
$V_i$). As $\gF_{ab}$ is semi-simple, we are done if we prove that $[\gF_{ab},X]$ is an element of
$\kp^\bot$. This is equivalent to the fact that $[\gF_{ab},X]$ sends any $u$ of $V$ into $\C v_0$.

Let $u=u_1+u_2$ the decomposition of $u\in V$ with $u_1\in V_1$ and $u_2\in V_2$. Let $Xu_1=w_1+\gl v_0$
with $w_1\in V_1$ and $Xu_2=w_2$ with $w_2\in V_2$. Then we have:
\begin{align*}
[\gF_{ab},X]u&=\gF_{ab}Xu_1+\gF_{ab}Xu_2-X\gF_{ab}u_1-X\gF_{ab}u_2\\
&=\gF_{ab}(w_1+\gl v_0+w_2)-aXu_1-bXu_2\\
&=aw_1+b\gl v_0+bw_2-aw_1-a\gl v_0-bw_2)=(b-a)\gl v_0
\end{align*}
\end{proof}

Consider for a while $G=Sl_n(\C)$ acting by the adjoint representation on its Lie algebra
$\sln(\C)$. The conormal to the orbit $G.X$ is the set of points
$$\ensemble{(Y,Z)\in\kg\x\kg}{[Y,Z]=0, \exists\, g\in G, Y=g.X}$$
If $Y$ is nilpotent regular, all $Z$ such that $[Y,Z]=0$ are nilpotent
and the conormal to the orbit is contained in the variety (\ref{equ:car0}).
If $X$ is nilpotent non regular, there exists always $Z$ semi-simple such that $[X,Z]=0$ and the
conormal to the orbit is {\sl not contained} in the variety
(\ref{equ:car0}).

Consider again $G=Gl_n(\C)$ acting on $\gln(\C)$. In the stratification $\left(\widetilde S_{(\gp,\kO)}\right)$, the stratum of a
nilpotent $X$ is the direct sum of the orbit $G.X$ and of the center $\kc$ of $\kg$. The conormal to the stratum of $X$ is the direct
sum of the center of $\kg$ and of the conormal to the orbit in $\sln(\C)$. So, the conormal to
the stratum of $X$ is contained in the set  (\ref{equ:car0}) if and only if $X$ is regular nilpotent.

Let  $P$ be as before the stability subgroup of $v_0\in\C^n$. The same calculation than the proof
of lemma \ref{lema} shows that the conormal to the $P$-orbit is the set
$$\ensemble{(Z,Y)\in\kg\x\kg}{Z=g.X, g\in P, \textrm{ and }[Z,Y]\in\kp^\bot}$$
while the conormal to the stratum of $X$ is the set
$$\ensemble{(Z,Y)\in\kg\x\kg}{Z=g.X+X_0, g\in P, X_0\in\kc, Y\notin\kc, [Z,Y]\in\kp^\bot}$$
This set is contained in the characteristic variety of $\MFP$ that is the set  (\ref{equ:cark}) if and
only if all non nilpotent points commuting with $X$ are in the center $\kc$.

So we have three options:

1) If the $P$-orbit of $X$ is dense in the $G$-orbit this means that the tangent vector fields
are the same hence that $\MFP$ and $\MFG$ are isomorphic in a neighborhood of $X$. As $\MFG$ is
tame (\cite[corollary 1.6.3]{GL}) the same is true for  $\MFP$.

2) If $X$ is nilpotent regular and the orbit $P.X$ is not dense in $G.X$, lemma \ref{lemb} shows that the
conormal to the stratum of $X$ is not contained in the characteristic variety of $\MFP$.

3) If $X$ is nilpotent non regular, the stratum of $X$ is not contained in the characteristic variety of $\MFP$
because the same was true for $\MFG$.

We proved:

\begin{cor} \label{propa}
Let $X$ be a nilpotent point of $\kg$. If the conormal to the direct sum of the center of $\kg$ and of the $P$-orbit
is contained in the characteristic variety of  $\MFP$, then $\MFP$ is isomorphic to $\MFG$ near $X$ and $\MFP$ admits a tame $b$-function.
\end{cor}

This was proved for $G=\Gln(\C)$ but extends immediately to the case where $G$ is a product $\prod Gl_{n_k}(\C)$

By the isomorphism $\gF^*$ of section \ref{sec:equiv}, this result gives an analogous result for $\NFG$
and in the next two sections we will consider the case of $\NFG$.

\subsection{Commutative algebra}\label{sec:comm}

As a second step of the proof, we assume that the rank of $[\kg,\kg]$ is $0$ which means that
$\kg$ is commutative. Hence $G=(\C^*)^N$ acting on $\C^n$ by componentwise multiplication. Then the
action of $G$ on $\kg\x V=\C^n\x\C^n$ is the multiplication on the second factor.

\begin{lem}
If $G=(\C^*)^N$ the module $\NFG$ is holonomic and tame.
\end{lem}

\begin{proof}
Let us fix coordinates $(x_1,\dots,x_n;y_1,\dots y_n)$ of $\kg\x V=\C^n\x\C^n$. The orbits of $G$
on $\kg\x V$ are given by the components of the normal crossing divisor $\{y_1y_2\dots y_n=0\}$ and
the vector fields tangent to the orbits are generated by $y_1D_{y_1},y_2D_{y_2},\dots y_nD_{y_n}$.

On the other hand, the set $F$ is a set of differential operators on $\kg$ whose principal symbols
define the zero section of the cotangent space to $\kg$. So the characteristic variety of the module $\NFG$ is the set:
\begin{equation*}
\ensemble{(x,y,\gx,\gh)\in T^*(\C^n\x\C^n)}{\gx_1=\dots=\gx_n=0,\quad y_1\gh_1=\dots=y_n\gh_n=0}
\end{equation*}
and the module is holonomic.

Define a stratification of $\C^n\x\C^n$ by the sets $\C^n\x S_\ga$ where the sets $S_\ga$ are the
smooth irreducible components of $\{y_1\dots y_n=0\}$ that is the sets $S_p=\{y_1=\dots =y_p=0,\ y_{p+1}\dots y_n\ne0\}$
and all the sets deduced by permutation of the $y_i$'s.

The characteristic variety of $\NFG$ is contained in the union of the conormals to the strata and
the operator $y_1D_{y_1}+y_2D_{y_2}+\dots+y_p D_{y_p}$ is a $b$-function for $S_p$ which is tame by
definition. So the module $\NFG$ is tame.
\end{proof}

\begin{defin}
If $\gS$ is a normal crossing divisor on a manifold $\gO$, we denote by $\CB_\gS$ the $\CaD$-module quotient
of $\DO$ by the ideal generated by the vector fields tangent to $\gS$.
\end{defin}

As the principal symbols of the differential operators of $F$ defines the zero section of the cotangent
space to $\kg$ the $\Dg$-module $\Dg/\Dg F$ is isomorphic to a power of $\CO_\kg$ \cite{BJORK} and $\NFG$
is isomorphic to a power of the module $\CB_\gS$ associated to $\{y_1\dots y_n=0\}$.

\subsection{Proof of the main theorem}

We will now prove theorem \ref{thm:main} by induction on the dimension of the semi-simple Lie algebra $[\kg,\kg]$.
More precisely, we will show the corresponding theorem for $\NFG$ which we know to be equivalent.

If the dimension $p$ of $[\kg,\kg]$ is $0$, the result has been proved in section \ref{sec:comm}. So we may
assume that $p$ is positive and that the result has been proved when the dimension is strictly lower than $p$.

Let $X=S+N$ be the Jordan decomposition of a point $X\in\kg$. If $S=0$ that is if $X$ is nilpotent, it
has been proved in section \ref{sec:nilpot} that the module  $\NFG$ is weakly tame along the stratum
going throw $X$ that is the orbit of $X$ plus the center.

So we may assume that $S\ne 0$ and consider the algebra $\kg^S$ that is the commutator of $S$. As $S$ is not
zero, $\kg^S$ is a reductive Lie algebra which is a direct sum of algebras $\kg\kl_{n_k}$. As the dimension
of  $[\kg^S,\kg^S]$ is strictly lower than $p$ the result is true for $\kg^S$.

We apply theorem \ref{thm:reg} to get a submersive map $\gY:G\x \km''\x V\to \kg\x V$ such that
 $\gY^*\NFG$ is equal to $\CO_G\hat\ox\CN_{F'\!\!,\,\km}$. Here $\km''$ is an open subset of
$\kg^S$ hence by the induction hypothesis $\CN_{F'\!\!,\,\km}$ is weakly tame and thus
$\gY^*\NFG$ is weakly tame.

As $\gY$ is submersive, this implies that $\NFG$ itself is weakly tame in a neighborhood of $S$. As it was
 remarked in the proof of \cite[Proposition 3.2.1.]{GL}, the stratum of $X=S+N$ meets any neighborhood of $S$ hence
the result is true in a neighborhood of $X$. This concludes the proof.

The hypersurface $\gS$ of $\kg$ was defined in remark \ref{rem:sigma} and by definition $\MFP$ is isomorphic to
$\MFG$ on $\kg\setminus \gS$.

\begin{prop}\label{prop:normal}
On the set $\kg_{rs}$ of regular semi-simple points, $\gS$ is a normal crossing divisor and $\MFP$ is isomorphic to a power of $\CB_\gS$.
\end{prop}

\begin{proof}
If $S$ is a regular semi-simple point, $\kg^S$ is a Cartan subalgebra of $\kg$ and the results of
\S\ref{sec:comm} may be applied. The module $\NFG$ is thus the inverse image by a submersion of a
power of the module associated to the normal crossing divisor $\{y_1y_2\dots y_n=0\}$. Hence $\NFG$
and by the isomorphism of \S\ref{sec:equiv} $\MFP$ are powers of the module associated to a normal
crossing divisor.
\end{proof}

The variety $\gS$ is the set of matrices $X$
such that $v_0,Xv_0,\dots,X^{n-1}v_0$ are linearly dependent. For example, if $v_0=(0,\dots,0,1)$, the equation
of $\gS$ is given by the determinant obtained by taking the last row of $I,X,\dots,X^{n-1}$.

\subsection{The $Sl_n(\C)$ case}

We consider $\sln(\C)$ as a component of the direct sum $\gln(\C)=\sln(\C)\oplus\C$.
When $Gl_n(\C)$ acts on $\gln(\C)$ the action is trivial on the center $\kc\simeq \C$ hence the set of vector fields
$\gt(\kg)$ are in fact defined on $\sln(\C)$ and are identical to the vectors induced by the action of $Sl_n(\C)$.
In the same way, if $P$ is the stability group in $ Gl_n(\C)$ of $v_0\in\C^n$ and $P'$ the stability group in
$Sl_n(\C)$ of the corresponding point of $\P_{n-1}(\C)$, $P'$ is the image of $P$ under the map $X\mapsto (det X)^{-1}X$.
So they define the same vector fields on $\sln(\C)$.

Let $F_0$ be the set of all vector fields on $\kc$.
If $F'$ is a (H-C)-type subset of $\Dgg$ for $\kg=\sln(\C)$, the set $F=F'\otimes F_0$ is a (H-C)-type for $Gl_n(\C)$ and we have
$$\MFP=\CM_{F'\!,\,\kp'}\otimes\CO_\kc$$

So the theorem \ref{thm:main} for $\Gln(\C)$ induces immediately the same theorem for $Sl_n(\C)$. The same
argument works for a product of copies of $\Gln(\C)$ and $Sl_n(\C)$.

Remark that a (H-C)-type subset of $\Dgg$ for $\kg=\gln(\C)$ is not the product of a  (H-C)-type subset
for $\sln(\C)$ by $F_0$ so we could not deduce the result for $\Gln(\C)$ from the corresponding result for
$Sl_n(\C)$. For the same reason if theorem  \ref{thm:main} is true for two Lie algebras this does not immediately
implies the result for their direct sum.

\subsection{Application to real forms}

Let $\kg_\R$ be $\sln(\R)$, $\gln(\R)$, $\sln(\C)$ or $\gln(\C)$ and $\kg$ be a complexification of $\kg_\R$, that is
$\sln(\C)$, $\gln(\C)$, $\ks\kl_{2n}(\C)$ or $\kg\kl_{2n}(\C)$ respectively. Let $\gS_\R$ be the intersection
of $\kg_\R$ with the variety $\gS$ of remark \ref{rem:sigma} and proposition \ref{prop:normal}.

If $U$ is an open subset of $\kg_\R^{rs}$, the set of semisimple regular points of $\kg_\R$,
$\gS_\R$ divides $U$ into a finite number of connected components $U_1,\dots,U_N$. Let $Y_i$
be the characteristic function of the open set $U_i$.

\begin{lem}
Let $U$ be a simply connected open subset of $\kg_\R^{rs}$. Any distribution $T$ solution on $U$ of $\MFP$
is equal to a finite sum $\sum f_i(x)Y_i(x)$ where $f_i$ is an analytic function defined on $U$ and solution of $\MFG$.
\end{lem}

\begin{proof}
On $U\setminus \gS_\R$, $\MFP$ is isomorphic to $\MFG$ hence $T$ is a solution of $\MFG$. By \cite{GL}, we know
that $\MFG$ is elliptic on $\kg_\R^{rs}$. Thus for each connected component $U_i$, $T|_{U_i}$ is an analytic
function solution of $\MFG$. Hence it extends to a solution of  $\MFG$ on the whole of $U$.

This shows that $T$ is equal to  $\sum f_i(x)Y_i(x)$ plus a distribution $S$ supported by $\gS_\R$. But $\MFP$
is weakly tame hence has no solutions supported by a hypersurface. So $S=0$ and  $T=\sum f_i(x)Y_i(x)$
\end{proof}

Let us now prove theorem \ref{thm:main2}.

Let $T$ be a distribution on an open subset of $\kg_\R$ which is solution of $\MFP$ and invariant under $P_\R$.
By the previous lemma the restriction of $T$ to  $\kg_\R^{rs}$ is a sum $\sum f_i(x)Y_i(x)$ where $f_i$ is an analytic
function defined on $U$ and solution of $\MFG$. But on the complement of $\gS_\R$ the orbits of $P_\R$ and $G_\R$ are
the same. Hence if $T$ is invariant under $P_\R$ all functions $f_i$ are equal and $T$ is an analytic solution of $\MFG$.

By  \cite[corollary 1.6.3]{GL}, $T|_{\kg_\R^{rs}}$ extends to a $L^1_{loc}$ function $T'$ on $\kg_\R$ solution of
$\MFG$. The distribution $S=T-T'$ is a distribution solution of $\MFP$ supported by the hypersurface
$\kg_\R\setminus\kg_{\R,rs}$ hence vanishes. This shows that $T$ is a  $L^1_{loc}$-function on $\kg_\R$ solution
of $\MFG$, hence that $T$ is $G$-invariant.

\begin{rem}
Solutions of $\MFP$ which are not globally invariant by $P_\R$ may not be solution of $\MFG$. As an
example, in the case of $\sld$ with the notations of \S\ref{sec:sld} the Heaviside function $Y(z)$
equal to $0$ if $z<0$ and to $1$ if $z\ge0$ is a solution of $\MFP$ but not of $\MFG$.
\end{rem}
\providecommand{\bysame}{\leavevmode\hbox to3em{\hrulefill}\thinspace}
\providecommand{\MR}{\relax\ifhmode\unskip\space\fi MR }
\providecommand{\MRhref}[2]{%
  \href{http://www.ams.org/mathscinet-getitem?mr=#1}{#2}
}
\providecommand{\href}[2]{#2}

\vespa\vespa

\enddocument

\end